\documentclass{amsart}

\usepackage{amssymb,amsthm,amstext,amsfonts,amsmath,yfonts,latexsym}

\renewcommand{\c}{\mathfrak{c}}
\newcommand{\filt}{\mathcal{F}}
\newcommand{\powerw}{\mathcal{P}(\omega)}
\newcommand{\inte}[2][X]{\mathrm{int}_{#1}(#2)}
\newcommand{\cl}[2][X]{\mathrm{cl}_{#1}(#2)}
\newcommand{\pair}[1]{\langle #1 \rangle}
\newcommand{\fr}{\mathcal{F}_r}
\renewcommand{\b}{\mathfrak{b}}
\newcommand{\B}{\mathcal{B}}
\newcommand{\C}{\mathcal{C}}
\newcommand{\Z}{\mathbb{Z}}
\newcommand{\MA}{\mathbf{MA}}

\newtheoremstyle{theorem}
     {11pt}
     {11pt}
     {}
     {}
     {\bfseries}
     {}
     {.5em}
     {\noindent\thmnumber{#2}. \thmname{#1}\thmnote{#3}}

\theoremstyle{theorem}

\newtheorem{ques}{Question}[section]
\newtheorem{lemma}[ques]{Lemma}
\newtheorem{propo}[ques]{Proposition}
\newtheorem{thm}[ques]{Theorem}

\title{Reversible filters}
\author[Dow]{Alan Dow}
\author[Hern\'andez-Guti\'errez]{Rodrigo Hern\'andez-Guti\'errez}
\address{Department of Mathematics and Statistics, University of North Carolina at Charlotte, Charlotte NC 28223}
\email[Dow]{adow@uncc.edu}
\email[Hern\'andez-Guti\'errez]{rodrigo.hdz@gmail.com}

\begin{document}

\begin{abstract}
A space is reversible if every continuous bijection of the space onto itself is a homeomorphism. In this paper we study the question of which countable spaces with a unique non-isolated point are reversible. By Stone duality, these spaces correspond to closed subsets in the \v Cech-Stone compactification of the natural numbers $\beta\omega$. From this, the following natural problem arises: given a space $X$ that is embeddable in $\beta\omega$, is it possible to embed $X$ in such a way that the associated filter of neighborhoods defines a reversible (or non-reversible) space? We give the solution to this problem in some cases. It is especially interesting whether the image of the required embedding is a weak $P$-set.
\end{abstract}

\maketitle

\section{Introduction}

A topological space $X$ is reversible if every time that $f:X\to X$ is a continuous bijection, then $f$ is a homeomorphism. This class of spaces was defined in \cite{revers_first}, where some examples of reversible spaces were given. These include compact spaces, Euclidean spaces $\mathbb{R}\sp{n}$ (by the Brouwer invariance of domain theorem) and the space $\omega\cup\{p\}$, where $p$ is an ultrafilter, as a subset of $\beta\omega$. This last example is of interest to us. 

Given a filter $\filt\subset\powerw$, consider the space $\xi(\filt)=\omega\cup\{\filt\}$, where every point of $\omega$ is isolated and every neighborhood of $\filt$ is of the form $\{\filt\}\cup A$ with $A\in\filt$. Spaces of the form $\xi(\filt)$ have been studied before, for example by Garc\'ia-Ferreira and Uzc\'ategi (\cite{filtspace1} and \cite{filtspace2}). When $\filt$ is the Fr\'echet filter, $\xi(\filt)$ is homeomorphic to a convergent sequence, which is reversible; when $\filt$ is an ultrafilter it is easy to prove that $\xi(\filt)$ is also reversible, as mentioned above. Also, in \cite[section 3]{hered-revers}, the authors of that paper study when $\xi(\filt)$ is reversible for filters $\filt$ that extend to precisely a finite family of ultrafilters (although these results are expressed in a different language).

Let us say that a filter $\filt\subset\powerw$ is reversible if the topological space $\xi(\filt)$ is reversible. It is the objective of this paper to study reversible filters. First, we give some examples of filters that are reversible and others that are non-reversible, besides the trivial ones considered above. Due to Stone duality, every filter $\filt$ on $\omega$ gives rise to a closed subset $K_\filt\subset \omega\sp\ast=\beta\omega\setminus\omega$ (defined below). Then our main concern is to try to find all possible topological types of $K_\filt$ when $\filt$ is either reversible or non-reversible. Our results are as follows.

\begin{itemize}
 \item Given any compact space $X$ embeddable in $\beta\omega$, there is a reversible filter $\filt$ such that $X$ is homeomorphic to $K_\filt$. (Theorem \ref{omegaastrev})
 \item Given any compact, extremally disconnnected space $X$ embeddable in $\beta\omega$, there is a non-reversible filter $\filt$ such that $X$ is homeomorphic to $K_\filt$. (Theorem \ref{EDnonrev})
 \item If $X$ is a compact, extremally disconnnected space that can be embedded in $\omega\sp\ast$ as a weak $P$-set and $X$ has a proper clopen subspace homeomorphic to itself, then there is a non-reversible filter $\filt$ such that $X$ is homeomorphic to $K_\filt$ and $K_\filt$ is a weak $P$-set of $\omega\sp\ast$. (Theorem \ref{weakPnotrev})
 \item There is a compact, extremally disconnnected space $X$ that can be embedded in $\omega\sp\ast$ as a weak $P$-set and every time $\filt$ is a filter with $X$ homeomorphic to $K_\filt$ and $K_\filt$ is a weak $P$-set, then $\filt$ is reversible. (Theorem \ref{example})
 \item Given any compact, extremally disconnected space $X$ that is a continuous image of $\omega\sp\ast$, there is a reversible filter $\filt$ such that $X$ is homeomorphic to $K_\filt$ and $K_\filt$ is a weak $P$-set of $\omega\sp\ast$. (Theorem \ref{weakPrev})
 \end{itemize}
 
 Also, in section \ref{MAsection}, using Martin's axiom, we improve some of the results above by constructing filters $\filt$ such that $K_\filt$ is a $P$-set.

\section{Preliminaries and a characterization}

Recall that $\beta\omega$ is the Stone space of all ultrafilters on $\omega$ and $\omega\sp\ast=\beta\omega\setminus\omega$ is the space of free ultrafilters. We will assume the reader's familiarity of most of the facts about $\beta\omega$ from \cite{vanmillhandbook}.  Recall that a space is an $F$-space if every cozero set is $C\sp\ast$-embedded. Since $\omega\sp\ast$ is an $F$-space we obtain some interesting properties. For example, every closed subset of $\omega\sp\ast$ of type $G_\delta$ is regular closed and every countable subset of $\omega\sp\ast$ is $C\sp\ast$-embedded. We will also need the more general separation property.

\begin{thm}\cite[3.3]{vdhandbook}\label{gaps}
Let $\B$ and $\C$ be collections of clopen sets of $\omega\sp\ast$ such that $\B\cup\C$ is pairwise disjoint, $|\B|<\b$ and $\C$ is countable. Then there exists a non-empty clopen set $C$ such that $\bigcup\B\subset C$ and $(\bigcup\C)\cap C=\emptyset$.
\end{thm}

We will be considering spaces embeddable in $\beta\omega$. There is no ZFC characterization of spaces embeddable in $\beta\omega$ but we have the following embedding results. A space is extremally disconnnected (ED, for short) if the closure of every open subset is open.

\begin{thm}
\begin{itemize}
 \item \cite[1.4.4]{vanmillhandbook} Under CH, any closed suspace of $\omega\sp\ast$ can be embedded as a nowhere dense $P$-set.
 \item \cite[1.4.7]{vanmillhandbook} Every compact, $0$-dimensional ED space of weight $\leq\c$ embedds in $\omega\sp\ast$.
 \item \cite[3.5]{vanmillhandbook}, \cite{dow-BN} If $X$ is an ED space and a continuous image of $\omega\sp\ast$, then $X$ can be embedded in $\omega\sp\ast$ as a weak $P$-set.
\end{itemize}
\end{thm}

Given $A\subset\omega$, we denote $\cl[\beta\omega]{A}\cap\omega\sp\ast=A\sp\ast$. Also, if $f:\omega\to\omega$ is any bijection, there is a continuous extension $\beta f:\beta\omega\to\beta\omega$ which is a homeomorphism; denote $f\sp\ast=\beta f\!\!\restriction_{\omega\sp\ast}$.

The Fr\'echet filter is the filter $\fr=\{A:\omega\setminus A\in[\omega]\sp{<\omega}\}$ of all cofinite subsets of $\omega$ and we will always assume that our filters extend the Fr\'echet filter. Each filter $\filt\subset\powerw$ defines a closed set $K_\filt=\{p\in\beta\omega:\filt\subset p\}$ that has the property that $A\in\filt$ iff $K_\filt\subset A\sp\ast$ and moreover, $K_\filt=\bigcap\{A\sp\ast:A\in\filt\}$.   Notice that $\xi(\filt)$ is the quotient space of $\omega\cup K_\filt\subset\beta\omega$ when $K_\filt$ is shrunk to a point. The first thing we will do is to find a characterization of reversible filters in terms of continuous maps of $\beta\omega$.

\begin{lemma}\label{revremainder}
Let $\filt$ be a filter on $\omega$. Then $\filt$ is not reversible if and only if there is a bijection $f:\omega\to\omega$ such that $f\sp\ast[K_\filt]$ is a proper subset of $K_\filt$.
\end{lemma}
\begin{proof}
First, assume that $g:\xi(\filt)\to \xi(\filt)$ is a continuous bijection that is not open. Then, $g[\omega]=\omega$ so let $f=g\!\!\restriction_{\omega}:\omega\to\omega$, which is a bijection.

Let $A\subset\omega$ such that $K_\filt\subset A\sp\ast$. Then $A\cup\{\filt\}$ is open, so by continuity of $g$ we obtain that $g\sp\leftarrow[A\cup\{\filt\}]=f\sp\leftarrow[A]\cup\{\filt\}$ is also open. Thus, $f\sp\leftarrow[A]\in\filt$ which implies that $K_\filt\subset f\sp\leftarrow[A]\sp\ast$. This implies that $f\sp\ast[K_\filt]\subset  A\sp\ast$. Thus, we obtain that
$$
f\sp\ast[K_\filt]\subset\bigcap\{A\sp\ast:K_\filt\subset A\sp\ast\}=K_\filt.
$$

Now, since $g$ is not open, there is $B\in\filt$ such that $f[B]\notin\filt$. Thus, $K_\filt\not\subset f[B]\sp\ast$. Since $f\sp\ast[K_\filt]\subset f[B]\sp\ast$, it follows that $K_f\not\subset f\sp\ast[K_\filt]$ so $K_f\neq f\sp\ast[K_\filt]$. We have proved that $f\sp\ast[K_\filt]\subsetneq K_\filt$.

Now, assume that $f:\omega\to\omega$ is a bijection such that $f\sp\ast[K_\filt]\subsetneq K_\filt$. Let $g=f\cup\{\pair{\filt,\filt}\}$, let us prove that this function is continuous but not open. 

We first prove that $g$ is continuous. Clearly, continuity follows directly for points of $\omega$ so let us consider neighborhoods of $\filt$ only. Any neighborhood of $\filt$ is of the form $A\cup\{\filt\}$ with $A\in\filt$. Then $K_\filt\subset A\sp\ast$ and $f\sp\ast[K_\filt]\subset A\sp\ast$ too, so $K_\filt\subset f\sp\leftarrow[A]\sp\ast$. This implies that $f\sp\leftarrow[A]\in\filt$. We obtain that $g\sp\leftarrow[A\cup\{\filt\}]=f\sp\leftarrow[A]\cup\{\filt\}$ is a neighborhood of $\filt$.

Now, let us prove that $g$ is not open. Since $f\sp\ast[K_\filt]\subsetneq K_\filt$, there exists $B\subset \omega$ such that $f\sp\ast[K_\filt]\subset B\sp\ast$ and $K_\filt\not\subset B\sp\ast$. But $K_\filt\subset f\sp\leftarrow[B]\sp\ast$. Then $f\sp\leftarrow[B]\cup\{\filt\}$ is a neighborhood of $\filt$ with image $g[f\sp\leftarrow[B]\cup\{\filt\}]=B\cup\{\filt\}$ that is not open.
\end{proof}

So from now on we will always use Lemma \ref{revremainder} when we want to check whether a filter is reversible.

According to \cite[Section 6]{revers_first}, a space is hereditarily reversible if each one of its subspaces is reversible. Given a filter $\filt$ on $\omega$, every subspace of $\xi(\filt)$ is either discrete or of the form $\xi(\filt\!\!\restriction_A)$ for some $A\in [\omega]\sp\omega$. Here $\filt\!\!\restriction_A=\{A\cap B:B\in\filt\}$. So call a filter $\filt$ hereditarily reversible if $\filt\!\!\restriction_A$ is reversible for all $A\in [\omega]\sp\omega$.

We present some characterizations of properties of $\filt$ and their equivalences for $K_\filt$. The proof of these properties is easy and left to the reader. 

\begin{lemma}\label{topfilters}
Let $\filt$ be a filter on $\omega$.
\begin{itemize}
\item[(a)] $\xi(\filt)$ is a convergent sequence if and only if $\filt=\fr$ if and only if $K_\filt=\omega\sp\ast$.
\item[(b)] $\xi(\filt)$ contains a convergent sequence if and only if $\inte[\omega\sp\ast]{K_\filt}\neq\emptyset$.
\item[(c)] $\xi(\filt)$ is Fr\'echet-Urysohn if and only if $K_\filt$ is a regular closed subset of $\omega\sp\ast$.
\item[(d)] $\filt$ is an ultrafilter if and only if $|K_\filt|=1$.
\item[(e)] Let $A\subset\omega$. Then $K_{\filt\!\restriction_A}= K_\filt\cap A\sp\ast$.
\end{itemize}
\end{lemma}

\section{First results}

From Lemma \ref{topfilters}, we can easily find all reversible filters that have convergent sequences. Notice that Proposition \ref{sequencenot} follows from \cite[Theorem 2.1]{hered-revers}. However, we include a proof to ilustrate a first use of Lemma \ref{revremainder}.

\begin{propo}\label{sequencenot}
Let $\filt$ be a filter on $\omega$ such that $\xi(\filt)$ has a convergent sequence. Then the following are equivalent
\begin{itemize}
 \item[(a)] $\filt$ is the Fr\'echet filter,
 \item[(b)] $\filt$ is hereditarily reversible, and
 \item[(c)] $\filt$ is reversible.
\end{itemize}
\end{propo}
\begin{proof}
From Lemma \ref{topfilters} we immediately get that (a) implies (b). That (b) implies (c) is clear so let us prove that (c) implies (a). So assume that $\filt\neq\fr$ has a convergent sequence. By Lemma \ref{topfilters} there is $A\subset\omega$ such that $A\sp\ast\subset K_\filt$. And since $\filt\neq\fr$, there is $B\in[\omega]\sp\omega$ with $\omega\setminus B\in\filt$. Thus, $K_f\cap{B}\sp\ast=\emptyset$. 

Let $A=A_0\cup A_1$ and $B=B_0\cup B_1$ be partitions into two infinite subsets. Now, let $f:\omega\to\omega$ be a bijection such that $f$ is the identity restricted to $\omega\setminus(A\cup B)$, $f[B_1]=B$, $f[B_0]=A_0$ and $f[A]=A_1$. Then it easily follows that $f\sp\ast[K_\filt]=K_\filt\setminus{A_0}\sp\ast\subsetneq K_\filt$, which shows that $\filt$ is not reversible by Lemma \ref{revremainder}.
\end{proof}

Clearly, every ultrafilter is hereditarily reversible by Lemmas \ref{revremainder} and \ref{topfilters} (this is known from \cite[Example 9]{revers_first}). By considering ultrafilters with different Rudin-Keisler types, we may find many other examples with isolated points. So naturally the question is whether there exists a reversible filter $\filt$ that is different from these examples. More precisely, we consider the following formulation of the problem.\vskip10pt

\begin{quote}
Let $X$ be a space that can be embedded in $\omega\sp\ast$ and consider a filter $\filt$ such that $K_f$ is homeomorphic to $X$. Is it possible to choose $\filt$ in such a way that $\filt$ is reversible? or not reversible?
\end{quote}\vskip10pt

For $X=\omega\sp\ast$, both questions have a positive answer. If $\filt=\fr$, then $K_\filt$ is homeomorphic to $\omega\sp\ast$ and $\filt$ is reversible. Now, say $\omega=A\cup B$ is a partition into infinite subsets and $\filt=\{C\subset A:|A\setminus C|=\omega\}$; then $K_\filt$ is homeomorphic to $\omega\sp\ast$ and $\filt$ is not reversible (Proposition \ref{sequencenot}). In the next result, we shall show that there are many reversible filters that are non-trivial and in fact, any closed subset of $\omega\sp\ast$ can be realized by one of them.

\begin{thm}\label{omegaastrev}
There exists a filter $\filt_0$ on $\omega$ with the following properties
\begin{itemize}
\item[(a)] any filter that extends $\filt_0$ is reversible,
\item[(b)] $K_{\filt_0}$ is crowded and nowhere dense, and
\item[(c)] if $X$ is any closed subset of $\omega\sp\ast$, there exists a filter $\filt\supset\filt_0$ such that $K_\filt$ is homeomorphic to $X$.
\end{itemize}
\end{thm}
\begin{proof}
Let $\{p_n:n<\omega\}\subset\omega\sp\ast$ be a sequence of weak $P$-points with different RK types; that such a collection exists follows from \cite{simon-indeplinkfam}. Let $\omega=\bigcup\{A_n:n<\omega\}$ be a partition into infinite subsets, we may assume that $p_n\in A_n\sp\ast$ for all $n<\omega$. Define $\filt_0$ to be the filter of all subsets $B\subset\omega$ such that there is $n<\omega$ with $B\cap A_m=\emptyset$, if $m\leq n$; and $B\cap A_m\in p_m$, if $m>n$. It is easy to see that $K_f=\cl[\omega\sp\ast]{\{p_n:n<\omega\}}\setminus{\{p_n:n<\omega\}}$, notice that this implies that $K_f$ is nowhere dense. Also, since every countable subset of $\omega\sp\ast$ is $C\sp\star$-embedded, it follows that $K_f$ is homeomorphic to $\omega\sp\ast$. From this, parts (b) and (c) follow.

So we are left to prove part (a). Let $\filt\supset\filt_0$ be any filter and let $f:\omega\to\omega$ be a bijection such that $f\sp\ast[K_\filt]\subset K_\filt$, according to Lemma \ref{revremainder} we have to prove that $f\sp\ast[K_\filt]= K_\filt$. Consider the set 
$$
B=\{n<\omega:f\sp\ast(p_n)\in\{p_k:k<\omega\}\}.
$$

Notice that $\{p_n:n<\omega\}$ and $\{f\sp\ast(p_n):n\in \omega\setminus B\}$ are disjoint sets of weak $P$-points of $\omega\sp\ast$. Thus, $\{p_n:n<\omega\}\cup\{f\sp\ast(p_n):n\in \omega\setminus B\}$ is a discrete set. But countable sets in an $F$-space such as $\omega\sp\ast$ are $C\sp\ast$-embedded so  $\cl[\omega\sp\ast]{\{p_n:n<\omega\}}\cap\cl[\omega\sp\ast]{\{f\sp\ast(p_n):n\in \omega\setminus B\}}=\emptyset$. Since $f\sp\ast[K_\filt]\subset K_\filt$, we obtain that $f\sp\ast[K_\filt]\cap\cl[\omega\sp\ast]{\{f\sp\ast(p_n):n\in \omega\setminus B\}}=\emptyset$. Thus, $K_\filt\subset\cl[\omega\sp\ast]{\{p_n:n\in B\}}$.  

From the fact that the ultrafilters chosen have different RK types, we obtain that $f\sp\ast(p_n)=p_n$ for all $n\in B$. From this it follows that in fact, $f$ restricted to $\cl[\omega\sp\ast]{\{p_n:n\in B\}}$ is the identity function. Thus, $f\sp\ast[K_\filt]= K_\filt$.
\end{proof}

Next we will produce a non-reversible filter $\filt$ with $K_\filt$ homeomorphic to any closed subset of $X$ that is ED. First, we will need two lemmas. Notice that an infinite, compact, $0$-dimensional and ED space $X$ has weight $\geq\c$. To see this, consider any pairwise disjoint family $\{U_n:n<\omega\}$ of pairwise disjoint clopen sets and for every $A\subset\omega$, let $V_A=\cl{\{U_n:n\in A\}}$, which is clopen. Then $\{V_A:A\subset\omega\}$ is a family of $\c$ different clopen subsets of $X$.

\begin{lemma}\label{directlimit}
Let $\{X_n:n<\omega\}$ be infinite, compact, $0$-dimensional, ED spaces of weight $\c$. Then there exists a $0$-dimensional, ED space $Y$ such that $Y=\bigcup\{Y_n:n<\omega\}$, where
\begin{itemize}
\item[(a)] $Y_n\subset Y_{n+1}$ whenever $n<\omega$, and
\item[(b)] $Y_n$ is homeomorphic to $X_n$ for each $n<\omega$.
\end{itemize}
Moreover, $Y$ is normal and has exactly $\c$ clopen sets.
\end{lemma}
\begin{proof}
Recall that in every $0$-dimensional, ED space, all countable subsets are $C\sp\ast$-embedded. Thus, every infinite, compact, $0$-dimensional, ED space has a copy of $\beta\omega$. Also, every compact, $0$-dimensional, ED space of weight at most $\c$ can be embedded in $\omega\sp\ast$. This implies that for every $n<\omega$, there exists a topological copy of $X_n$ embedded in $X_{n+1}$.

So for each $n<\omega$, let $e_n:X_n\to X_{n+1}$ an embedding. If $n\leq m<\omega$, denote by $e_n\sp{m}:X_n\to X_m$ the composition of all such appropriate embeddings. In the union $\bigcup_{n<\omega}{(X_n\times\{n\})}$, define an equivalence relation $\pair{x,n}\sim\pair{y,m}$ and $n\leq m$ if and only if $y=e_n\sp{m}(x)$. So let $Y$ be the quotient space under this relation and for each $n<\omega$, let $Y_n$ be the image of $X_n\times\{n\}$ under the corresponding quotient map. It is easy to see that each $Y_n$ is homeomorphic to $X_n$ for each $n<\omega$. Notice that a set $U$ is open in $Y$ if and only if $U\cap Y_n$ is open in $Y_n$ for all $n<\omega$.

First, let us see that $Y$ is normal and $0$-dimensional. In fact, we will argue that if $F$ and $G$ are disjoint closed subsets of $Y$, they can be separated by a clopen subset. For each $n<\omega$, let $F_n=F\cap Y_n$ and $G_n=G\cap Y_n$. Since $F_0\cap G_0=\emptyset$ and $Y_0$ is compact and $0$-dimensional, there is a clopen set $U_0\subset Y_0$ with $F_0\subset U_0$ and $G_0\cap U_0=\emptyset$. Assume that $k<\omega$ and for each $n\leq k$ we have found $U_n$ clopen in $Y_n$ such that if $n\leq k$, then $F_n\subset U_n$, $G_n\cap U_n=\emptyset$ and if $n\leq m<k$ then $U_m\cap Y_n=U_n$. Now, the two sets $F_{k+1}\cup U_k$ and $G_{k+1}\cup (Y_k\setminus U_k)$ are disjoint and closed in $Y_{k+1}$. Then choose a clopen subset $U_{k+1}$ such that $F_{k+1}\cup U_k\subset U_{k+1}$ and $[G_{k+1}\cup (Y_k\setminus U_k)]\cap U_{k+1}=\emptyset$. This concludes the recursive construction of $\{U_n:n<\omega\}$. Finally, let $U=\bigcup\{U_n:n<\omega\}$, notice that $F\subset U$ and $G\cap U=\emptyset$. Also, $U$ is clopen because $U\cap Y_n=U_n$ is clopen in $Y_n$ for each $n<\omega$.

To see that $Y$ is ED, let $U\subset Y$ be open, we have to prove that $\cl[Y]{U}$ is clopen. We will define a sequence of open sets $U_\alpha\subset\cl[Y]{U}$ for all ordinals $\alpha$. Let $U_0=U$ and if $\alpha$ is a limit ordinal, define $U_\alpha=\bigcup_{\beta<\alpha}{U_\beta}$. Now assume that $U_\alpha$ is defined and let $U_{\alpha+1}=\bigcup\{\cl[Y_n]{U_\alpha\cap Y_n}:n<\omega\}$. Since $Y_n$ is closed in $Y$ for every $n<\omega$, $U_{\alpha+1}\subset\cl[Y]{U}$. Moreover, $Y_n$ is ED so $\cl[Y_n]{U_\alpha\cap Y_n}$ is open in $Y_n$ for each $n<\omega$. Also, clearly $\cl[Y_n]{U_\alpha\cap Y_n}\subset \cl[Y_m]{U_\alpha\cap Y_m}$ whenever $n<m<\omega$. From this it follows that $U_{\alpha+1}$ is open and we have finished our recursive construction. Notice that $U_\alpha\subset U_\beta$ whenever $\alpha<\beta$. So there exists some $\gamma<|Y|\sp+$ such that $U_\gamma=U_{\gamma+1}$. 

Notice that for all $n<\omega$, $\cl[Y_n]{U_\gamma\cap Y_n}\subset U_{\gamma+1}\cap Y_n=U_\gamma\cap Y_n$ so in fact $U_\gamma\cap Y_n$ is clopen in $Y_n$. From this it follows that $U_\gamma$ is clopen. Since $U\subset U_\gamma\subset\cl[Y]{U}$, we obtain that $U_\gamma=\cl[Y]{U}$. Then $Y$ is ED.

Since every clopen set $U$ of $Y$ is a union of the clopen subsets $U\cap Y_n$, for $n<\omega$, it follows that there are at most $\c$ clopen subsets of $Y$. Also, since $Y$ is normal, $Y_0$ is $C\sp\ast$-embedded in $Y$ so $Y$ has at least $\c$ clopen sets. This completes the proof.
\end{proof}

\begin{lemma}\label{lemmaembedding}
Let $\{A_n:n<\omega\}$ be pairwise disjoint infinite subsets of $\omega$ and for each $n<\omega$, let $K_n$ be a closed subset of $A_n\sp\ast$. Then $\bigcup\{K_n:n<\omega\}$ is $C\sp\ast$-embedded in $\beta\omega$.
\end{lemma}
\begin{proof}
Let $f:\bigcup\{K_n:n<\omega\}\to[0,1]$ be a continuous function. Given $n<\omega$, since $K_n$ is closed in $\beta A_n$, there is a function $g_n:A_n\to [0,1]$ such that $\beta g_n\!\!\restriction_{K_n}=f_{K_n}$. So if $g:\omega\to[0,1]$ is any function extending $\bigcup\{g_n:n<\omega\}$, then $\beta g:\beta\omega\to[0,1]$ is an extension of $f$.
\end{proof}

\begin{thm}\label{EDnonrev}
Let $X$ be any compact, $0$-dimensional, ED space of weight $\leq\c$. Then there is a non-reversible filter $\filt$ on $\omega$ such that $K_\filt$ is homeomorphic to $X$.
\end{thm}

\begin{proof}
Let $\{X_n:n<\omega\}$ be a family of pairwise disjoint clopen subsets of $X$. Let $B\subset\omega$ with $|B|=|\omega\setminus B|=\omega$, let $\{A_n:n\in\Z\}$ be a partition of $\omega\setminus B$ into infinite subsets and let $f:\omega\to\omega$ be a bijection such that $f\!\!\restriction_B$ is the identity function in $B$ and for all $n\in\Z$, $f[A_n]=A_{n+1}$.

By Lemma \ref{directlimit}, there is an $0$-dimensional, ED space $Y$ with exactly $\c$ clopen sets that is equal to the increasing union of spaces $\{Y_n:n<\omega\}$ such that $Y_n$ is homeomorphic to $X_n$ and $Y_n\subset Y_{n+1}$ for all $n<\omega$. Recall that $\beta Y$ is also ED (\cite[1.2.2(a)]{vanmillhandbook}). Also, in a compact and $0$-dimensional space the weight is equal to the number of clopen sets so $\beta Y$ has weight $\c$. Thus, there is an embedding $e:Y\to A_0\sp\ast$ (\cite[1.4.7]{vanmillhandbook}).

Let $e_0=e$ and if $n<\omega$, let $e_{n+1}=f\sp\ast\circ e_n:Y\to A_{n+1}\sp\ast$. For each $n<\omega$, let $Z_n=e_n[X_n]$. Define $Z=\bigcup\{Z_n:n<\omega\}$ and let $W$ be a subset of $B\sp\ast$ homeomorphic to the set $X\setminus\cl{\bigcup{\{X_n:n<\omega\}}}$. Notice that $\bigcup\{X_n:n<\omega\}$ is a $C\sp\ast$-embedded subset of $X$ because $X$ is ED and $\bigcup\{Z_n:n<\omega\}$ is $C\sp\ast$-embedded in $Z$ by Lemma \ref{lemmaembedding}. Thus, there is an embedding $h:\cl{\bigcup{\{X_n:n<\omega\}}}\to\omega\sp\ast$ such that $h[X_n]=Z_n$ for all $n<\omega$. Since $X$ is extremally disconnected, we may extend $h$ to an embedding $H:X\to\omega\sp\ast$ in such a way that $H[X\setminus\cl{\bigcup{\{X_n:n<\omega\}}}]=W$.

So let $\filt$ be the filter of all $A\subset\omega$ with $Z\cup W\subset A\sp\ast$. We will prove that $\filt$ is not reversible by showing that $f\sp\ast[Z\cup W]\subsetneq Z\cup W$. First, notice that $f\sp\ast[W]=W$ and $f\sp\ast[Z_n]=e_{n+1}[X_n]\subset Z_{n+1}$ for all $n<\omega$. Finally, $f\sp\ast[Z\cup W]\cap A_0\sp\ast=\emptyset$ so $f\sp\ast[Z\cup W]\cap Z_0=\emptyset$. This completes the proof.
\end{proof}

\section{Embedding as weak $P$-sets}

Recall that every ED space that is a continuous image of $\omega\sp\ast$ can be embedded in $\omega\sp\ast$ as a weak $P$-set ( \cite[3.5]{vanmillhandbook}, \cite{dow-BN}). So now we study a problem similar to the one in the previous section, adding the requirement that the embedded space is a weak $P$-set. More carefully stated, we want the following.\vskip10pt

\begin{quote}
Let $X$ be a space that can be embedded in $\omega\sp\ast$ as a weak $P$-set and consider a filter $\filt$ such that $K_f$ is a weak $P$-set homeomorphic to $X$. Is it possible to choose $\filt$ in such a way that $\filt$ is reversible? or not reversible?
\end{quote}\vskip10pt

First, we start finding filters that are not reversible. The construction is similar to that in Theorem \ref{EDnonrev}. However, it needs an extra hypothesis.

\begin{thm}\label{weakPnotrev}
 Let $X$ be a compact ED space that can be embedded in $\omega\sp\ast$ as a weak $P$-set. Moreover, assume that there exists a proper clopen subspace of $X$ homeomorphic to $X$. Then there is a non-reversible filter $\filt$ on $\omega$ such that $K_\filt$ is a weak $P$-set homeomorphic to $X$.
\end{thm}
\begin{proof}
 From the hypothesis on $X$, it is easy to find a collection of non-empty, pairwise disjoint clopen sets $\{X_n:n<\omega\}$ of $X$ that are pairwise homeomorphic. Let $B\subset\omega$ with $|B|=|\omega\setminus B|=\omega$, let $\{A_n:n\in\Z\}$ be a partition of $\omega\setminus B$ into infinite subsets and let $f:\omega\to\omega$ be a bijection such that $f\!\!\restriction_B$ is the identity function in $B$ and for all $n\in\Z$, $f[A_n]=A_{n+1}$.
 
 It is not hard to argue that there is an embedding $e:\bigcup\{X_n:n<\omega\}\to\bigcup\{A_n\sp\ast:n<\omega\}$ in such a way that for each $n<\omega$, $e[X_n]$ is a weak $P$-set of $A_n\sp\ast$ and $f\sp\ast[e[X_n]]=e[X_{n+1}]$. Since $\bigcup\{X_n:n<\omega\}$ is $C\sp\ast$-embedded in $X$ and $X$ is ED, we may assume that $X\subset\omega\sp\ast$, $e$ is the identity function and $X\setminus\cl{\bigcup\{X_n:n<\omega\}}$ is a weak $P$-set of $B\sp\ast$. 
 
 Now let us see that with these conditions, $X$ is in fact a weak $P$-set. Let $\{x_n:n<\omega\}$ be disjoint from $X$. Then for each $n<\omega$, $X_n$ is a weak $P$-set so $\cl[\omega\sp\ast]{\{x_n:n<\omega\}}\cap X_n=\emptyset$. Thus, the family $\{X_n:n<\omega\}\cup\{\{x_n\}:n<\omega\}$ is discrete and countable so it can be separated by pairwise disjoint clopen sets. By Lemma \ref{lemmaembedding}, it easily follows that $\bigcup\{X_n:n<\omega\}$ can be separated from $\{x_n:n<\omega\}$ by a continuous function. Also, $\cl[\omega\sp\ast]{\{x_n:n<\omega\}}\cap (X\setminus\cl{\bigcup\{X_n:n<\omega\}})=\emptyset$. So in fact $\cl[\omega\sp\ast]{\{x_n:n<\omega\}}\cap X=\emptyset$, which is what we wanted to prove.
 
 Finally, let $\filt$ be the neighborhood filter of $X$ so that $K_\filt=X$. It remains to notice that $f\sp\ast[X]\subset X\setminus A_0\sp\ast\subsetneq X$. Thus, the statement of the theorem follows. 
\end{proof}

Next, we would like to show that the extra hypothesis of Theorem \ref{weakPnotrev} is really necessary. 

\begin{thm}\label{example}
There exists a compact ED space $X$ that can be embedded in $\omega\sp\ast$ as a weak $P$-set and such that every time $\filt$ is a filter with $K_\filt$ a weak $P$-set homeomorphic to $X$ then $\filt$ is reversible.
\end{thm}
\begin{proof}
 In \cite{dow-gubbi-szymanski} it was shown that there exists a separable, ED, compact space $X$ that is rigid in the sense that the identity function is its only autohomeomorphism. Using very similar arguments, it can be easily proved that no clopen subset of $X$ is homeomorphic to $X$. Since $X$ is separable and crowded, it is easy to see that $X$ is a continuous image of $\omega\sp\ast$. This in turn implies that $X$ can be embedded in $\omega\sp\ast$ as a weak $P$-set.
 
 Assume now that $\filt$ is any filter on $\omega$ such that $K_\filt$ is a weak $P$-set homeomorphic to $X$. Let $f:\omega\to\omega$ be a bijection such that $f\sp\ast[K_\filt]\subset K_\filt$ and assume that $U=K_\filt\setminus f\sp\ast[K_\filt]\neq\emptyset$. Then, since $X$ is separable, there is a countable set $D\subset U$ with $\cl[\omega\sp\ast]{D}=\cl[K_\filt]{U}$. Since $D\cap f\sp\ast[K_\filt]=\emptyset$ and $f\sp\ast[K_\filt]$ is a weak P-set, it follows that $\cl[\omega\sp\ast]{D}\cap f\sp\ast[K_\filt]=\emptyset$. Thus, $\cl[K_\filt]{U}\cap f\sp\ast[K_\filt]=\emptyset$ which shows that $U=\cl[K_\filt]{U}$ and $f\sp\ast[K_\filt]$ is clopen in $K_\filt$. So $f\sp\ast[K_\filt]$ is a clopen set of $K_\filt$ homeorphic to itself, which implies $f\sp\ast[K_\filt]=K_\filt$. This is a contradiction so in fact $U=\emptyset$ and $f\sp\ast[K_\filt]=K_\filt$. This shows that $\filt$ is reversible.
\end{proof}

We finally consider filters that are reversible. In order to make the corresponding spaces weak $P$-sets of $\omega\sp\ast$, we will need to use Kunen's technique of a construction of a weak $P$-point (\cite{kunen-weakP}). We shall use Dow's approach from \cite{dow-BN}. 

First, let us recall the concept of a $\c$-OK set. So let $\kappa$ be an infinite cardinal, $X$ a space and $K$ closed in $X$. Given an increasing sequence $\{C_n:n<\omega\}$ of closed subsets of $X$ disjoint from $K$, we will say that $K$ is $\kappa$-OK with respect to $\{C_n:n<\omega\}$ if there is a set $\mathfrak{U}$ of neighborhoods of $K$ such that $|\mathfrak{U}|=\kappa$ and every time $0<n<\omega$ and $\mathfrak{U}_0\in[\mathfrak{U}]\sp{n}$, $\bigcap\mathfrak{U}_0\cap C_n=\emptyset$. Then $K$ is $\kappa$-OK if it is $\kappa$-OK with respect to every countable increasing sequence of closed subsets of $X$. It easily follows that if a closed set is $\kappa$-OK for $\kappa$ uncountable, then it is a weak $P$-set (see \cite[Lemma 1.3]{kunen-weakP}).

 In \cite[Lemma 3.2]{dow-BN}, Dow proves that if $\omega\sp\ast$ maps onto $X$, then there is an continuous surjection $\varphi:\omega\sp\ast\to X\times(\c+1)\sp\c$, where $\c+1=\c\cup\{\c\}$ is taken as the one-point compactification of the discrete space $\c$. This map $\varphi$ will replace Kunen's independent matrices from \cite{kunen-weakP}. Lemma 3.4 in \cite{dow-BN} gives a method to construct $\c$-OK points in $\omega\sp\ast$ using this map $\varphi$. We will use the following modification mentioned by Dow by the end of the proof of \cite[Theorem 3.5]{dow-BN}. For any set $I\subset\c$, we denote by $\pi_{I}:X\times (\c+1)\sp\c\to X\times (\c+1)\sp{I}$ the projection. To be consistent with notation, $\pi_\emptyset$ will denote the projection of $X\times (\c+1)\sp\c\to X$. 
 
 \begin{lemma}\label{OKlemma}
  Let $\psi:\omega\sp\ast\to X\times(\c+1)\sp I$ be continuous and onto, where $I\subset\c$ is an infinite set. Assume that $K\subset\omega\sp\ast$ is a closed set with $\psi[K]=X\times(\c+1)\sp I$ and $\{C_n:n<\omega\}$ is a sequence of closed subsets of $\omega\sp\ast$ disjoint from $K$. Then there is a countable set $J\subset I$ and a closed subset $K\sp\prime\subset K$ such that
  \begin{itemize}
   \item $(\pi_{I\setminus J}\circ\psi)[K\sp\prime]=X\times(\c+1)\sp{I\setminus J}$, and
   \item $K\sp\prime$ is $\c$-OK with respect to $\{C_n:n<\omega\}$.
  \end{itemize}
 \end{lemma}
 
  Recall in $(\c+1)\sp{J}$, where $J$ any set, there is a base of clopen subsets of the form $\prod\{U_\xi:\xi\in J\}$ where each factor $U_\xi$ is clopen and the support $\{\xi\in J: U_\xi\neq \c+1\}$ is finite.
 
\begin{thm}\label{weakPrev}
 Let $X$ be a compact ED space that is a continuous image of $\omega\sp\ast$. Then there is a reversible filter $\filt$ such that $K_\filt$ is a weak $P$-set homeomorphic to $X$.
\end{thm}
\begin{proof}
 Let $\varphi:\omega\sp\ast\to X\times(\c+1)\sp\c$ be the surjection from \cite[Lemma 3.2]{dow-BN}. Our objective is to recursively construct a closed set $K\subset\omega\sp\ast$ such that $(\pi_\emptyset\circ\varphi)[K]=X$ and $(\pi_\emptyset\circ\varphi)\!\!\restriction_K:K\to X$ is irreducible. By a classic result by Gleason (see, for example, the argument in \cite[1.4.7]{vanmillhandbook}) it follows that $(\pi_\emptyset\circ\varphi)\!\!\restriction_K$ is a homeomorphism. So it only remains to take $\filt$ to be the filter of neighborhoods of $K$. 
 
 We will define $K$ as the intersection of a family $\{K_\alpha:\alpha<\c\}$ of closed subsets of $\omega\sp\ast$, ordered inversely by inclusion. We will also define a decreasing sequence $\{I_\alpha:\alpha<\c\}\subset\c$ such that $I_0=\c$ and $|\c\setminus I_\alpha|<|\alpha|\cdot\omega$ for all $\alpha<\c$. We will have the following conditions:
 
 \begin{itemize}
  \item[(a)] If $\beta<\c$ is a limit, $K_\beta=\bigcap\{K_\alpha:\alpha<\c\}$ and $I_\beta=\bigcap\{I_\alpha:\alpha<\c\}$.
  \item[(b)] For each $\alpha<\c$, $(\pi_{I_\alpha}\circ\varphi)[K_\alpha]=X\times(\c+1)\sp{I_\alpha}$.
 \end{itemize}

 We need to do acomplish three things in our construction: make $K$ a weak $P$-set, that the map $(\pi_\emptyset\circ\varphi)\!\!\restriction_K:K\to X$ is irreducible and make sure that the filter of neighborhoods $\filt$ is reversible. So we will partition ordinals into three sets. For $i\in\{0,1,2\}$, let $\Lambda_i$ be the set of ordinals $\alpha<\c$ such that $\alpha=\beta+n$, $\beta$ is a limit ordinal and $n<\omega$ is congruent to $i$ modulo $3$. Let $\{\{C_n\sp\alpha:n<\omega\}:\alpha\in\Lambda_0\}$ be an enumeration of all countable increasing of clopen sets where each sequence is repeated cofinally often. Let $\{B_\alpha:\alpha\in\Lambda_1\}$ be an enumeration of all clopen subsets of $\omega\sp\ast$. For these two types of steps we need the following conditions.
 
 \begin{itemize}
  \item[(c)] Let $\alpha\in\Lambda_0$. If $K_\alpha$ is disjoint from all the members of the sequence $\{C_n\sp\alpha:n<\omega\}$, then $|I_\alpha\setminus I_{\alpha+1}|\leq\omega$ and $K_{\alpha+1}$ is $\c$-OK with respect to $\{C_n\sp\alpha:n<\omega\}$.
  \item[(d)] Let $\alpha\in\Lambda_1$. If $(\pi_{I_\alpha}\circ\varphi)[K_\alpha\cap B_\alpha]=X\times(\c+1)\sp{I_\alpha}$, then $I_{\alpha+1}=I_\alpha$ and $K_{\alpha+1}=K_\alpha\cap B_\alpha$. Otherwise, there are clopen sets $C\subset X$ and $D\subset(\c+1)\sp{\c}$ such that the support of $D$ is equal to $I_\alpha\setminus I_{\alpha+1}$, $\varphi[K_\alpha\cap B_\alpha]\cap(C\times D)=\emptyset$ and $K_{\alpha+1}=K_\alpha\cap\varphi\sp\leftarrow[X\times D]$.
 \end{itemize}

 Clearly, (c) follows from Lemma \ref{OKlemma} and implies that $K$ is a weak $P$-set of $\omega\sp\ast$. Also, it is not hard to see that condition (d) implies that $(\pi_\emptyset\circ\varphi)\!\!\restriction_K:K\to X$ is irreducible. Conditions (c) and (d) are taken from the proof of \cite[Theorem 3.5]{dow-BN}.
 
 Finally, we need to take care of reversibility using the $\c$ chances we get from $\Lambda_2$. Let $\{f_\alpha:\alpha\in\Lambda_2\}$ be an enumeration of all bijections from $\omega$ onto itself, each one repeated cofinally often. We will require the following condition.
 
 \begin{itemize}
  \item[(e)] Let $\alpha\in\Lambda_2$. Assume that there exists a clopen sets $U\subset\omega\sp\ast$ and $V\subset X$ such that $(\pi_{I_\alpha}\circ\varphi)[K_\alpha\cap U]=V\times(\c+1)\sp{I_\alpha}$ and $(\pi_\emptyset\circ\varphi)[f_\alpha\sp\ast[K_\alpha\cap U]]\subset X\setminus V$. Then $|I_\alpha\setminus I_{\alpha+1}|<\omega$ and there is $x\in X$ such that $f_\alpha\sp\ast[K_{\alpha+1}\cap (\pi_\emptyset\circ\varphi)\sp\leftarrow(x)]\cap K_{\alpha+1}=\emptyset$.
 \end{itemize}
 
 Before we show how to prove that (e) can be obtained, let us show why it implies that the filter of neighborhoods of $K$ is reversible. Assume that after our construction, $K$ is not reversible. Then by Lemma \ref{revremainder}, there is a bijection $f:\omega\to\omega$ such that $f\sp\ast[K]\subsetneq K$. By property (d) above we know that $(\pi_\emptyset\circ\varphi)\!\!\restriction_K:K\to X$ is irreducible so $(\pi_\emptyset\circ\varphi\circ f\sp\ast)[K]$ is a proper closed subset of $X$. Let $V$ be a clopen set disjoint from $(\pi_\emptyset\circ\varphi\circ f\sp\ast)[K]$. Now consider the clopen subset $W=(\pi_\emptyset\circ\varphi)\sp\leftarrow[V]$ of $\omega\sp\ast$. From the definition of $K$ and the facts that $f\sp\ast$ is a homeomorphism and $f\sp\ast[K]\cap W=\emptyset$, there is $\beta<\c$ such that $f\sp\ast[K_\gamma]\cap W=\emptyset$ every time $\beta\leq\gamma<\c$. So fix $\alpha\in\Lambda_2$ such that $\beta\leq\alpha$ and $f_\alpha=f$. Now define
 $$
 U=W\cap ((f\sp\ast)\sp{\leftarrow}[\omega\sp\ast\setminus W]),
 $$
 which is a clopen set of $\omega\sp\ast$ with the property that $(\pi_\emptyset\circ\varphi\circ f\sp\ast)[U]\subset X\setminus V$. From the choice of $\alpha$ we obtain that $K_\alpha\cap W\subset U$. Also, $(\pi_{I_\alpha}\circ\varphi)[K_\alpha\cap W]=V\times(\c+1)\sp{I_\alpha}$ by property (b). Finally, notice that $K_\alpha\cap W=K_\alpha\cap U$. Thus, our choice of $\alpha$, $U$ and $V$ satisfy the hypothesis of condition (e). So let $x$ as in the conclusion of (e) and take $p\in K$ such that $(\pi_\emptyset\circ\varphi)(p)=x$. Then $p\in K_{\alpha+1}\cap (\pi_\emptyset\circ\varphi)\sp\leftarrow(x)$ implies that $f_\alpha\sp\ast(p)\notin K_{\alpha+1}$. Thus, $p\in K\setminus f\sp\ast[K]$, a contradiction. This contradiction comes from the fact that we assumed that $K$ was not reversible.
 
 So we are left to prove that condition (e) can be achieved. So assume we have $\alpha$, $U$ and $V$ like in the hypothesis of (e). Choose any $i\in I_\alpha$ and let $J=I_\alpha\setminus\{i\}$. For each $\xi\in\c$, let $U_\xi=(\pi_{\{i\}}\circ\varphi)\sp\leftarrow(X\times\{\xi\})$, which is a clopen set in $X$. Then $\{U_\xi:\xi\in\c\}$ is a pairwise disjoint collection of clopen subsets of $X$ such that $(\pi_{J}\circ\varphi)[K_\alpha\cap U_\xi]=X\times(\c+1)\sp{J}$ for all $\xi\in\c$ . For each $\xi\in\c$, consider the set $V_\xi=(f_\alpha\sp\ast)\sp\leftarrow[U_\xi]\cap K_\alpha\cap U$, which is a clopen set of $K_\alpha\cap U$. Here we will have two cases.\vskip10pt

 \noindent{\it Case 1:} There exists $\xi_0\in\c$ such that $(\pi_J\circ\varphi)[V_{\xi_0}]=V\times(\c+1)\sp{J}$. Choose any $\xi_1\in\c\setminus\{\xi_0\}$, let $I_{\alpha+1}=J$ and define
 $$
 K_{\alpha+1}=V_{\xi_0}\cup(K_\alpha\cap U_{\xi_1}\cap(\pi_\emptyset\circ\varphi)\sp\leftarrow[X\setminus V]).
 $$
 Notice that $K_{\alpha+1}\subset K_\alpha$ and $(\pi_{I_{\alpha+1}}\circ\varphi)[K_{\alpha+1}]=X\times(\c+1)\sp{I_{\alpha+1}}$. Now take any $x\in V$. Then $K_{\alpha+1}\cap(\pi_\emptyset\circ\varphi)\sp\leftarrow(x)\subset V_{\xi_0}$ so $f\sp\ast[K_{\alpha+1}\cap(\pi_\emptyset\circ\varphi)\sp\leftarrow(x)]\subset U_{\xi_0}$. Since $U_{\xi_0}\cap U_{\xi_1}=\emptyset$, then $f\sp\ast[K_{\alpha+1}\cap(\pi_\emptyset\circ\varphi)\sp\leftarrow(x)]\cap K_{\alpha+1}=\emptyset$ and we are done.\vskip10pt
 
 \noindent{\it Case 2:} Not Case 1. Take $\xi_0\in\c$, then there exists clopen sets $C\subset V$ and $D\subset(\c+1)\sp{J}$ such that $C\times D$ is disjoint from $(\pi_J\circ\varphi)[V_{\xi_0}]$. Let $J\sp\prime\subset J$ be the support of $D$ and define $I_{\alpha+1}=J\setminus J\sp\prime$. In this case, define
 $$
 K_{\alpha+1}=(K_\alpha\cap U\cap(\pi_{J}\circ\varphi)\sp\leftarrow[C\times D])\cup(K_\alpha\cap U_{\xi_0}\cap(\pi_{J}\circ\varphi)\sp\leftarrow[(X\setminus C)\times D]).
 $$
 Clearly, $K_{\alpha+1}\subset K_\alpha$. It is not hard to see that and $(\pi_J\circ\varphi)[K_{\alpha+1}]=X\times D$, which in turn implies that $(\pi_{I_{\alpha+1}}\circ\varphi)[K_{\alpha+1}]=X\times (\c+1)\sp{I_{\alpha+1}}$. Now let $x\in C$. Assume there is $p\in K_{\alpha+1}$ such that $(\pi_{\emptyset}\circ\varphi)(p)=x$ and $q=f_\alpha\sp\ast(p)\in K_{\alpha+1}$, we will reach a contradiction. Notice that $p\in U$, which implies that $q\in (\pi_{\emptyset}\circ\varphi)\sp\leftarrow[X\setminus V]$. So from the definition of $K_{\alpha+1}$ we obtain that $q\in U_{\xi_0}$. This in turn implies that $p\in V_{\xi_0}$. By the choice of $C\times D$ we obtain that $(\pi_J\circ\varphi)(p)\notin C\times D$. But since $x\in C$, $p\in K_{\alpha+1}\cap U\cap(\pi_{J}\circ\varphi)\sp\leftarrow[C\times D])$ so $(\pi_{J}\circ\varphi)(p)\in C\times D$. So we obtain a contradiction and we obtain the negation of our assumption. Thus, $f\sp\ast[K_{\alpha+1}\cap(\pi_\emptyset\circ\varphi)\sp\leftarrow(x)]\cap K_{\alpha+1}=\emptyset$, which is what we wanted to prove.\vskip10pt
 
 These two cases complete the proof of condition (e) and finish the proof of the Theorem.
\end{proof}

\section{Filters generated by towers}\label{MAsection}

It is well known that Martin's axiom (henceforth, $\MA$) implies that there are filters that are $P$-filters (see, for example, \cite[Theorem 4.4.5]{bart}). Equivalently, there is a filter $\filt$ such that $K_\filt$ is a $P$-set. It is not hard to see that by changing all instances of ``weak $P$-set'' to just ``$P$-set'' in Theorem \ref{weakPnotrev}, we obtain a valid statement. Also, every $P$-set in a weak $P$-set so the $P$-set version of Theorem \ref{example} is in fact implied by Theorem \ref{example}. 

As for the $P$-set version of Theorem \ref{weakPrev}, we will do something stronger, but only for separable, first countable spaces. Recall that a tower is a set $\{A_\alpha:\alpha<\kappa\}\subset\mathcal{P}(\omega)$, for some $\kappa$, such that
\begin{itemize}
 \item $A_\beta\setminus A_\alpha$ is finite every time $\alpha<\beta<\kappa$, and 
 \item there is no $A\subset\omega$ such that $A\setminus A_\alpha$ is finite for every $\alpha<\kappa$.
\end{itemize}
In this case, $\{A_\alpha\sp\ast:\alpha<\kappa\}$ is a decreasing chain of clopen subsets of $\omega\sp\ast$ with nowhere dense intersection. Clearly, every tower generates a filter and every filter generated by a tower of height $\kappa=\c$ is a $P$-filter. In fact, every filter generated by a tower of height $\c$ is a $P_{\c}$-filter. 

In what follows we will assume the reader's familiarity with $\MA$ and small uncountable cardinals from \cite{vdhandbook}. One fact that we will use several times is that $\MA$ implies that every intersection of less than $\c$ many clopen sets is a regular closed set (this follows from Theorem \ref{gaps}).

\begin{lemma}\label{lifting}
 Let $X_0$, $X_1$ be compact separable spaces of weight $<\c$ and let $\psi_0:\omega\sp\ast\to X_0$ be a continuous onto function. Assume that there is a continuous function $\pi:X_1\to X_0$ and a partition $X_1=V_0\cup V_1$ into two clopen sets such that $\pi\!\!\restriction_{V_i}:V_i\to X_0$ is an embedding for $i<2$. Then $\MA$ implies there exists a clopen set $W\subset\omega\sp\ast$ and a continuous onto function $\psi_1:W\to X_1$ such that $\pi\circ\psi_1=\psi_0\!\!\restriction_W$.
\end{lemma}
\begin{proof}
 Let $F_i=\pi[V_i]$ for $i<2$, this is a closed subset of $X_0$. Choose a countable dense set $\{d_n:n<\omega\}$ of $X_0$ that is contained in the dense open set $(X_0\setminus F_0)\cup(X_0\setminus F_1)\cup(\inte[X_0]{F_0\cap F_1})$. Let $N_i=\{n<\omega:d_n\in X_0\setminus F_i\}$ for $i<2$ and $N_2=\omega\setminus(N_0\cup N_1)$.
 
 Since $F_0$ is an intersection of $<\c$ many clopen sets of $X_0$, there is a collection $\mathcal{G}_0$ of clopen sets of $\omega\sp\ast$ such that $\bigcap\mathcal{G}_0=\psi_0\sp\leftarrow[F_0]$ and $|\mathcal{G}_0|<\c$. For each $n\in N_0\cup N_2$, let $U_n$ be a clopen set of $\omega\sp\ast$ such that $\psi_0[U_n]=\{d_n\}$. Clearly, $U_n\subset\bigcap\mathcal{G}_0$ for all $n\in N_0\cup N_2$. Thus, considering the collection $\mathcal{G}_0\cup\{U_n:n\in N_0\cup N_2\}$, by $\MA$ and Lemma \ref{gaps}, there exists an clopen set $W_0\subset\omega\sp\ast$ such that $W_0\subset\bigcap\mathcal{G}_0$ and $U_n\subset W_0$ for all $n\in N_0\cup N_2$. It follows that $W_0$ is a clopen set of $\omega\sp\ast$ such that $d_n\in \psi[W_0]$ for all $n\in N_0\cup N_2$. Since $\{d_n:n\in N_0\cup N_2\}$ is dense in $F_0$ we obtain that $\psi_0[W_0]=F_0$.
 
 Now we will find a clopen set $W_1$ with $\psi_0[W_1]=F_1$. However, we will have to be more careful because of possible intersections with $W_0$. Let $\mathcal{G}_1$ be a collection of clopen sets of $\omega\sp\ast$ such that $\bigcap\mathcal{G}_1=\psi_0\sp\leftarrow[F_1]$ and $|\mathcal{G}_1|<\c$. For each $n\in N_1$, let $U_n$ be a clopen subset of $\omega\sp\ast$ such that $\psi_0[U_n]=\{d_n\}$. If $n\in N_2$, we choose two disjoint non-empty clopen subsets $U_n\sp{0}$ and $U_n\sp{1}$ of $\omega\sp\ast$ such that $U_n\sp{0}\subset W_0$ and $\psi_0[U_n\sp{i}]=\{d_n\}$ for $i<2$. Clearly, $U_n\subset\bigcap\mathcal{G}_1$ for $n\in N_1$ and $U_n\sp{0}\cup U_n\sp{1}\subset \bigcap\mathcal{G}_1$ for $n\in N_2$. So using $\MA$ and Lemma \ref{gaps} again, we can find a clopen set $W_1\subset\omega\sp\ast$ such that $W_1\subset\bigcap\mathcal{G}_1$, $U_n\subset W_1$ for all $n\in N_1$, $U_n\sp{1}\subset W_1$ for all $n\in N_2$  and $U_n\sp{0}\cap W_1=\emptyset$ for all $n\in N_2$. Again, it easily follows that $\psi_0[W_1]=F_1$.
 
 So now consider $W_0\setminus W_1$. If $n\in N_1$, $U_n\subset W_0\setminus W_1$ so $d_n\in \psi_0[W_0\setminus W_1]$. If $n\in N_2$, $U_n\sp{0}\subset W_0\setminus W_1$ so $d_n\in \psi_0[W_0\setminus W_1]$. From this it follows that $\psi_0[W_0\setminus W_1]=F_0$. Let $W=W_0\cup W_1$, we now define $\psi_1:W\to X_1$ such that
 $$
 \psi_1(x)=\left\{
 \begin{array}{cl}
 (\pi\!\!\restriction_{V_0})\sp{-1}(\psi_0(x)),  & \textrm{ if }x\in W_0\setminus W_1,\textrm{ and}\\
 (\pi\!\!\restriction_{V_1})\sp{-1}(\psi_0(x)),  & \textrm{ if }x\in W_1.
 \end{array}
 \right.
 $$
 
 It is easy to see that $\psi_1$ is as required. 
\end{proof}

\begin{thm}\label{MAthm}
 Let $X$ be a separable, compact, ED space. Then $\MA$ implies that there is a reversible filter $\filt$ that is generated by a tower of height $\c$ such that $K_\filt$ is a $P$-set homeomorphic to $X$.
\end{thm}

\begin{proof}
 By our hypothesis, we may assume that $X\subset{}\sp\c{2}$. For all pairs $\alpha\leq\beta\leq\c$, let $\pi_{\alpha}\sp\beta:{}\sp\beta{2}\to{}\sp\alpha{2}$ be the projection. Let $\{d_n:n<\omega\}$ be an enumeration of a countable dense set in $X$. By permuting the elements of $\c$ and then renaming them if necessary, we may assume that if $n,m<\omega$ and $\pi\sp\c_\omega(d_n)=\pi\sp\c_\omega(d_m)$, then $n=m$. Let $X_\alpha=\pi\sp\c_\alpha[X]$ for every $\alpha<\c$.
 
 We will recursively construct a decreasing sequence $\{K_\alpha:\omega\leq\alpha<\c\}$ of clopen sets of $\omega\sp\ast$ and a sequence of continuous functions $\varphi_\alpha:K_\alpha\to X_\alpha$, for $\omega\leq\alpha\leq\c$, in such a way that $\pi\sp\beta_\alpha\circ\varphi_\beta=\varphi_\alpha$ whenever $\omega\leq\alpha\leq\beta<\c$. Once we have done this, let $K=\bigcap\{K_\alpha:\alpha<\c\}$ and define $\varphi:K\to X$ by $\varphi(x)=\bigcup\{\varphi_\alpha(x):\omega\leq\alpha<\c\}$ for all $x\in K$. Notice that $\varphi$ is continuous and $\pi\sp\c_\alpha\circ\varphi=\varphi_\alpha$ for every $\omega\leq\alpha<\c$. 
 
 Let $\{B_\alpha:\omega\leq\alpha<\c\}$ be an enumeration of all clopen subsets of $\omega\sp\ast$. Let $\{f_\alpha:\omega\leq\alpha<\c\}$ be the collection of all bijections from $\omega$ onto itself such that each one is repeated cofinally often. Let $\Lambda_0$ be the set of infinite even ordinals $<\c$ and let $\Lambda_1$ be the set of infinite odd ordinals $<\c$. We will carry out our construction respecting the following conditions.
 
 \begin{itemize}
  \item[(a)] $K_\omega=\omega\sp\ast$.
  \item[(b)] For all $\omega\leq\alpha<\c$ and $n<\omega$, $\varphi_\alpha[K_\alpha]=X_\alpha$.
  \item[(c)] For all $\alpha\in\Lambda_0$, if $\varphi_\alpha[K_\alpha\cap B_\alpha]=X_\alpha$, then $K_{\alpha+1}\subset B_\alpha$.
  \item[(d)] Let $\alpha\in\Lambda_1$. Assume that there exists a clopen sets $U\subset K_\alpha$ and $V\subset X_\alpha$ such that $\varphi_\alpha[U]=V$ and $\varphi_\alpha[f_\alpha\sp\ast[U]]\subset X\setminus V$. Then there is $x\in X_\alpha$ such that $f_\alpha\sp\ast[K_{\alpha+1}\cap\varphi_\alpha\sp\leftarrow(x)]\cap K_{\alpha+1}=\emptyset$.
 \end{itemize}
 
 It is not hard to prove that (c) implies that $\varphi:K\to X$ is irreducible, thus, a homeomorphism. And the proof that (d) implies that the filter $\filt$ of neighborhoods of $K$ is reversible is analogous to the corresponding one in the proof of Theorem \ref{weakPrev}. Since any separable subspace of $\omega\sp\ast$ is nowhere dense, we obtain that $\{K_\alpha:\alpha<\c\}$ is a tower that generates $\filt$. So we will only show how to carry out this construction. 
 
 Let $\beta<\c$ be a limit ordinal, let us show how to find $K_\beta$ and $\varphi_\beta$. Let $T=\bigcap\{K_\alpha:\alpha<\beta\}$ and define $\psi:T\to X_\beta$ by $\psi(x)=\bigcup\{\varphi_\alpha(x):\omega\leq\alpha<\beta\}$ for all $x\in T$. Notice that $\psi$ is continuous and $\MA$ implies that $T$ is a regular closed set of $\omega\sp\ast$. Because $X_\beta$ has weight $\leq|\beta|<\c$, $\psi\sp\leftarrow[\pi_\beta\sp{\c}(d_n)]$ is an intersection of $<\c$ many clopen sets for each $n<\omega$. By $\MA$, we can choose for each $n<\omega$ a clopen set $U_n\subset T$ such that $\psi[U_n]=\{\pi_\beta\sp\c(d_n)\}$. Then by considering the sets $\{U_n:n<\omega\}\cup\{K_\alpha:\alpha<\beta\}$, by $\MA$ and Lemma \ref{gaps}, we obtain that there is a clopen set $V\subset T$ such that $U_n\subset V$ for every $n<\omega$. Let $K_\alpha=V$ and $\varphi_\beta=\psi\!\!\restriction_{V}$.
 
 Now assume that $\alpha<\c$ and we want to define $K_{\alpha+1}$ and $\varphi_{\alpha+1}$. First, assume that $\alpha\in\Lambda_0$. Let $T=K_\alpha\cap B_\alpha$ if $\varphi_\alpha[K_\alpha\cap B_\alpha]=X_\alpha$; otherwise, let $T=K_\alpha$. Then $\varphi_\alpha[T]=X_\alpha$. Notice that $V_0=\{x\in X_{\alpha+1}:x(\alpha)=i\}$ for $i<2$ is a pair of clopen sets of $X_{\alpha+1}$ where $\pi\sp{\alpha+1}_\alpha$ is one-to-one and $X_{\alpha+1}=V_0\cup V_1$. Thus, we can apply Lemma \ref{lifting} to find a clopen set $W\subset T$ and a continuous function $\psi:W\to X_{\alpha+1}$ such that $\pi\sp{\alpha+1}_\alpha\circ\psi=\varphi_\alpha$. So let $K_\alpha=W$ and $\varphi_{\alpha+1}=\psi$.

 Finally, assume that $\alpha\in\Lambda_1$. If the hypothesis of (d) does not hold, just use Lemma \ref{lifting} like in the previous paragraph to define $K_{\alpha+1}$ and $\varphi_{\alpha+1}$. So assume otherwise. By $\MA$ and the fact that all points of $X_\alpha$ have character $\leq|\alpha|<\c$, we may assume that for each $n<\omega$, there exists a clopen set $U_n\subset K_\alpha$ such that $\varphi_\alpha[U_n]=\{\pi\sp{c}_\alpha(d_n)\}$. Let $N_0$ be the set of $n<\omega$ such that $d_n\in V$. We may assume that $U_n\subset U$ for all $n\in N_0$. For $n\in\omega\setminus N_0$, we may assume that either $U_n\subset f_\alpha\sp\ast[U]$ or $U_n\cap f_\alpha\sp\ast[U]=\emptyset$. Let $N_1$ the set of all $n\in\omega\setminus N_0$ such that $U_n\subset f_\alpha\sp\ast[U]$ and $N_2=\omega\setminus(N_0\cup N_1)$.
 
 For each $n\in N_1$, choose $p_n\in U_n$. Then $\{p_n:n\in N_1\}$ is a discrete (possibly empty) set contained in $f_\alpha\sp\ast[U]$. For each $n\in N_1$, let $q_n=(f_\alpha\sp\ast)\sp\leftarrow(p_n)$. Then $\{q_n:n\in N_1\}$ is a discrete set contained in $U$. Since no clopen subset of $\omega\sp\ast$ is separable, it is possible to choose, for each $n\in N_0$, $p_n\in U_n\setminus\cl[\omega\sp\ast]{\{q_m:m\in N_1\}}$. Then the set $\{p_n:n\in N_0\}\cup\{q_n:n\in N_1\}$ is a dicrete subset of $U$. Then, since countable subsets are $C\sp\ast$-embedded in $\omega\sp\ast$, there exists a clopen set $W\subset U$ such that $\{p_n:n\in N_0\}\subset W$ and $\{q_n:n\in N_1\}\cap W=\emptyset$. With this, we can define
 $$
 T=W\cup (K_\alpha\cap f_\alpha\sp\ast[U\setminus W])\cup (\varphi_\alpha\sp\leftarrow[X\setminus V]\cap (K_\alpha\setminus f_\alpha\sp\ast[U]))
 $$
 
 Clearly, $T$ is a clopen subset of $K_\alpha$. If $n\in N_0$, then $p_n\in W$ so $d_n\in\varphi_\alpha[T]$. If $n\in N_1$, then $p_n\in K_\alpha\cap f_\alpha\sp\ast[U\setminus W]$ so $d_n\in\varphi_\alpha[T]$. Finally, if $n\in N_2$, $U_n\subset \varphi_\alpha\sp\leftarrow[X\setminus V]\cap (K_\alpha\setminus f_\alpha\sp\ast[U])$ so $d_n\in\varphi_\alpha[T]$. Thus, $\{d_n:n<\omega\}\subset\varphi_\alpha[T]$, which implies that $\varphi_\alpha[T]=X$.
 
 By an application of Lemma \ref{lifting}, there is a clopen set $T\sp\prime\subset T$ and a continuous function $\psi:T\sp\prime\to X_{\alpha+1}$ such that $\pi\sp{\alpha+1}_\alpha\circ\psi=\varphi_\alpha$. Let $K_{\alpha+1}=T\sp\prime$ and $\varphi_{\alpha+1}=\psi$. Finally, choose $x\in V$ arbitrarily. Since $K_{\alpha+1}\cap\varphi_\alpha\sp\leftarrow(x)\subset T\cap\varphi_\alpha\sp\leftarrow(x)\subset W$, $f_\alpha\sp\ast[K_{\alpha+1}\cap\varphi_\alpha\sp\leftarrow(x)]\subset f_\alpha\sp\ast[W]\subset \omega\sp\ast\setminus T\subset \omega\sp\ast\setminus K_{\alpha+1}$. Thus, these choices satisty the conclusion of (d), so we have finished the proof.
\end{proof}

\begin{ques}
 Is the conclusion of Theorem \ref{MAthm} still valid if $X$ is not necessarily separable?
\end{ques}

\section*{Acknowledgements}

We would like to thank the referee for his or her remarks and for pointing out a mistake in a previous version of the proof of Theorem \ref{MAthm}.

\end{document}